\documentclass[10pt,reqno]{amsart}
\usepackage{amsmath,amsopn,amssymb,amsthm}
\usepackage[cp1251]{inputenc}
\usepackage[english]{babel}
\usepackage[pagebackref,breaklinks=true,colorlinks=true,linkcolor=blue,citecolor=blue,urlcolor=blue]{hyperref}
\usepackage{graphicx}%
\usepackage{color}

\newtheorem{theorem}{Theorem}[section]

\newtheorem{Definition}[theorem]{Definition}

\newtheorem{lemma}[theorem]{Lemma}

\newtheorem{Prop}[theorem]{Proposition}

\renewenvironment{proof}[1][Proof]{\noindent\textbf{#1.} }{\ \rule{0.5em}{0.5em}}

\newcommand{\m}{{\mathfrak m}}
\newcommand{\g}{{\mathfrak g}}
\newcommand{\h}{{\mathfrak h}}
\begin{document}

\title{Cyclic Finser metrics on homogeneous spaces}
\author{Ju Tan}
\address[Ju Tan]{School of Microelectronics and Data Science,
Anhui University of Technology, Maanshan, 243032, P.R. China}\email{tanju2007@163.com}
\author{Ming Xu$^*$}
\address[Ming Xu] {School of Mathematical Sciences,
Capital Normal University,
Beijing 100048,
P.R. China}
\email{mgmgmgxu@163.com}

\thanks{$^*$Ming Xu is the corresponding author.}

\begin{abstract}
In this paper, we generalize the notion of cyclic metric to homogeneous Finsler geometry. Firstly, we prove that a homogeneous Finsler space $(G/H, F)$ must be symmetric when it satisfies the naturally reductive and cyclic conditions simultaneously. Then we prove that a Finsler cyclic Lie group which is either flat or nilpotent must have an Abelian Lie algebra. Finally, we show how to induce a cyclic $(\alpha,\beta)$ metric from a cyclic Riemannian metric. Using this method, we construct a Randers cyclic Lie group.

\vbox{}
\end{abstract}

\maketitle

\section{Introduction}
We call a homogeneous Riemannian manifold $(G/H,\mathrm{g})$ {\it cyclic} with respect to a given reductive decomposition $\mathfrak{g}=\mathfrak{h}+\mathfrak{m}$, if the following property is satisfied
\begin{eqnarray}\label{thefirst}
 \langle [x, y]_{\m}, z \rangle + \langle [y, z]_{\m}, x \rangle + \langle [z, x]_{\m}, y\rangle =0, \quad\forall x, y, z \in \m.
 \end{eqnarray}
Here $\langle\cdot,\cdot\rangle$ is the $\mathrm{Ad}(H)$-invariant inner product on
$\mathfrak{m}=T_{eH}(G/H)$ determined by the $G$-invariant metric $\mathrm{g}$.
When $H$ is trivial (and then the reductive decomposition must be $\g=\h+\m=0+\g$), we call the cyclic $(G,\mathrm{g})$ a {\it cyclic Lie group}.
Moreover, if $G$
is unimodular, we specify the cyclic condition as the \textit{traceless cyclic} condition.

The motivation for studying cyclic metrics can be traced back to Tricerri-Vanhecke's classification for homogeneous structures \cite{TV1983}.
In their classification, there are three basic ones, i.e., $\mathcal{S}_1$, $\mathcal{S}_2$ and $\mathcal{S}_3$.
The cyclic (traceless cyclic) condition corresponds to the type $\mathcal{S}_1\oplus\mathcal{S}_2$ ($\mathcal{S}_2$ respectively).
Compared with the naturally reductive metrics, which have been more extensively studied and correspond to the type $\mathcal{S}_3$,
cyclic metrics have very different properties, worthy to be explored.

Here we briefly survey some related progress.
Kowalski
and Tricerri studied the traceless cyclic condition in \cite{KT1987}.
Gadea, Gon\'{a}lez-D\'{a}vila and Oubi\~{n}a concerned traceless cyclic Lie groups in \cite{GGO2015}, and characterized cyclic and traceless cyclic metrics in \cite{GGO2016}. In above literatures, relatively few examples have been discovered. Only those with low dimensions have been classified \cite{Bi1997,KT1987}.
Meanwhile, this project was generalized to homogeneous pseudo-Riemannian geometry in \cite{GO1992}. Calvaruso and L\'{o}pez classified Lorentzian cyclic Lie groups of dimension 3 and 4 \cite{Calvaruso}.
Tohidfa and Zaeim classified cyclic Lie groups of the signature type $(2,2)$ \cite{TZ2021}.


In this paper, we explore the generalization of the cyclic property in Finsler geometry,
replacing (\ref{thefirst}) with the following analog,
\begin{equation}\label{002}
      \langle [x, y]_{\m}, z \rangle_y + \langle [y, z]_{\m}, x \rangle_y + \langle [z, x]_{\m}, y\rangle_y =0, \quad\forall x,z \in \m,y\in\m\backslash\{0\}.
\end{equation}
Here the fundamental tensor $\langle\cdot,\cdot\rangle_y$ depends on $y$, so the Finsler cyclic condition (\ref{002}) is not multi-linear. Though this becomes an essential obstacle, we can still prove
some properties and find some examples for cyclic Finsler metrics.

Firstly, we prove

\begin{theorem}\label{main-thm-1}
If the homogeneous Finsler manifold is cyclic and naturally reductive
with respect to a given reductive decomposition, then that decomposition is a Cartan decomposition, i.e., that homogeneous Finsler manifold is a symmetric space.
\end{theorem}

The phenomenon in Theorem \ref{main-thm-1}
was pointed out in \cite{TV1983} when the metric is Riemannian.
It implies that, in homogeneous Finsler geometry,
the subclass of cyclic Finsler metrics have a trivial overlap with that of naturally reductive metrics, i.e.,
they generalize symmetric Finsler spaces in totally different directions.

Secondly, we generalize Proposition 5.5 and Proposition 3.4 in \cite{GGO2015} as follows.

\begin{theorem}\label{main-thm-4}
Non-Abelian nilpotent Lie groups do not admit cyclic left-invariant Finsler metrics.
\end{theorem}

\begin{theorem}\label{main-thm-2}
If $F$ is a left invariant cyclic Finsler metric on a Lie group $G$ which has constant flag curvature $K\equiv0$, then $G$ has an abelian Lie algebra.
\end{theorem}

Finally, we consider the construction of cyclic Finsler metrics.
The following theorem implies that
cyclic Randers and $(\alpha,\beta)$ metrics may be induced from those Riemannian ones.

\begin{theorem}\label{main-thm-3}
Let $G/H$ be a homogeneous manifold with a given reductive decomposition $\g=\h+\m$, $\alpha$
an $\mathrm{Ad}(H)$-invariant inner product on $\m$, and $\beta=\alpha(X,\cdot)$ in which $X\in\m\backslash\{0\}$ is an $\mathrm{Ad}(H)$-invariant vector. Suppose that $\alpha$ and $F=\alpha\phi(\tfrac{\beta}{\alpha})$ induce a cyclic Riemannian metric and a non-Riemannian homogeneous Douglas metric on $G/H$ respectively. Then $(G/H, F)$ is cyclic if and only if
\begin{eqnarray*}
  \alpha([y, u]_{\m}, y)[ \alpha(y,y)\alpha(X,v)- \alpha(X,y)\alpha(v,y)]=\alpha([y, v]_{\m}, y)[ \alpha(y,y)\alpha(X,u)-  \alpha(X,y)\alpha(u,y)]
\end{eqnarray*}
is satisfied for any $u,v,y\in\m$.
\end{theorem}
As the application, we constructed left invariant non-Riemannian cyclic Randers metrics on some solvable Lie groups.


\section{Preliminary in homogeneous Finsler geometry}

Let $G/H$ be a homogeneous manifold, which is endowed with a reductive decomposition $\g=\h+\m$.
Here the reductiveness means that the decomposition is $\mathrm{Ad}(H)$-invariant (it implies that, in the Lie algebra level, $[\mathfrak{h},\mathfrak{m}]\subset\mathfrak{m}$). Then the tangent space $T_{o}(G/H)$ at the origin $o=eH$ can be identified as $\m$, so that the isotropy action coincides with
the $\mathrm{Ad}(H)$-action on $\m$.

Any $G$-invariant Finsler metric $F$ on $G/H$ can be one-to-one determined by $F=F(o,\cdot)$, which is any arbitrary $\mathrm{Ad}(H)$-invariant Minkowski norm on $\mathfrak{m}$ \cite{De2012}. We call the pair $(G/H,F)$ a {\it homogeneous Finsler manifold}. For example, a homogeneous $(\alpha,\beta)$ metric can be determined by a Minkowski norm $F=\alpha\phi(\tfrac{\beta}{\alpha})$ on $\mathfrak{m}$, in which $\alpha$ is an $\mathrm{Ad}(H)$-invariant Euclidean norm, $\beta=\alpha(X,\cdot)$ is the $\alpha$-dual of some $\mathrm{Ad}(H)$-invariant $X\in\mathfrak{m}$, and $\phi(r)$ is some positive smooth function \cite{DH2013}.

A homogeneous metric $F$ on $G/H$ is called {\it naturally reductive} with respect to the given reductive decomposition $\g=\h+\m$, if
$$\langle[x,u]_\mathfrak{m},v\rangle_y+\langle[x,v]_\mathfrak{m},u\rangle_y+
2C_y([x,y]_\mathfrak{m},u,v)=0,$$
for any $y\in\m\backslash\{0\},x,u,v\in\mathfrak{m}$ \cite{La2007}. Here
$$\langle u,v\rangle_y=\tfrac12 \tfrac{\partial^2}{\partial s\partial t}|_{s=t=0}F^2(y+su+tv),\quad
C_y(u,v,w)=\tfrac12\tfrac{d}{dt}|_{t=0}\langle u,v\rangle_{y+tw}$$
are the {\it fundamental tensor} and the {\it Cartan tensor} of the Minkowski norm $F$ on $\mathfrak{m}$
respectively. See \cite{DH2010} for the equivalent definition and description of this natural reductiveness.

The {\it spray vector field} $\eta:\m\backslash\{0\}\rightarrow \m$ and the {\it connection operator} $N:(\m\backslash\{0\})\times\m\rightarrow\m$ for a homogeneous Finsler manifold $(G/H,F)$ with respect to a given reductive decomposition $\g=\h+\m$
are determined by the following equations respectively \cite{Hu2015-1},
\begin{eqnarray}
\langle \eta(y),u\rangle_y &=&\langle y,[u,y]_\m\rangle_y, \quad\forall y\in\m\backslash\{0\},\nonumber \\
2\langle N(y,v),u\rangle_y&=&\langle [u,v]_\m,y\rangle_y
+\langle[u,y]_\m,v\rangle_y+\langle[v,y]_m,u\rangle_y\nonumber\\
& &-2C_y(u,v,\eta(y)),\quad\forall y\in\m\backslash\{0\}, u,v\in\m.\label{008}
\end{eqnarray}
They are useful for presenting curvature formulae of homogeneous Finsler manifolds.
For example, we have the following homogeneous Riemann curvature formulae.
\cite{Hu2015-1,Hu2015-2}.

\begin{theorem} \label{thm-homogeneous-Riemann-curvature}
For a left invariant Finsler metric $F$ on the Lie group $G$,
the Riemann curvature $R_y:\g=T_eG\rightarrow T_eG=\g$ for any $y\in\g\backslash\{0\}$
can be presented as
\begin{equation}\label{009}
R_y(u)=D_\eta N(y,u)-N(y,N(y,u))+N(y,[y,u])-[y,N(y,u)].
\end{equation}
\end{theorem}
%

More details on the curvatures in general Finsler geometry can be found in \cite{BCS2000}.

\section{Cyclic Finsler metric and its properties}
Now we are ready to present the precise definition for the Finsler cyclic condition.

\begin{Definition}
A homogeneous Finsler manifold $(G/H, F)$ is said to be \textit{cyclic} with respect to a given reductive decomposition $\g=\h+\m$,  if we have
\begin{equation*}
\langle[x, y]_{\m} , z\rangle_y+
\langle[y, z]_{\m} ,  x\rangle_y+
\langle[z, x]_{\m} , y\rangle_y=0,\quad \forall y\in\m\backslash\{0\}, x, z\in \m.
\end{equation*}
In particular, when $H=\{e\}$ (and then the reductive decomposition must be $\g=\h+\m=0+\g$), we call the cyclic $(G,F)$ a ({\it{Finsler}}) {\it cyclic Lie group}.
\end{Definition}

\begin{proof}[Proof of Theorem \ref{main-thm-1}]
Suppose the homogeneous Finsler metric $F$ is naturally reductive and cyclic with respect to the reductive decomposition $\g=\h+\m$. Choose any $y\in\m\backslash\{0\}$, $x,z\in\m$.
Then the natural reductiveness provides
\begin{equation}\label{003}
\langle[x,u]_\m,v\rangle_y+\langle[x,v]_\m,u\rangle_y+2C_y([x,y]_\m,u,v)=0,\quad\forall u,v\in\m.
\end{equation}
Input $u=y$ and $v=z$ into (\ref{003}), we get
\begin{equation}\label{004}
\langle[x,y]_\m,z\rangle_y+g_y([x,z]_\m,y)=0.
\end{equation}
Similar argument also provides
\begin{equation}\label{007}
\langle[z,y]_\m,x\rangle_y+g_y([z,x]_\m,y)=0.
\end{equation}
The cyclic condition provides
\begin{equation}\label{005}
\langle[x,y]_\m,z\rangle_y+\langle[y,z]_\m,x\rangle_y+\langle[z,x]_\m,y\rangle_y=0.
\end{equation}
The sum of (\ref{007}) and (\ref{005}) minus (\ref{004}) provides
$
3\langle [z,x]_\m,y\rangle_y=0
$.
So we get $\langle[\mathfrak{m},\mathfrak{m}]_\m,y\rangle_y=0$, $\forall y\in\m\backslash\{0\}$.
It can only happen when $[\mathfrak{m},\mathfrak{m}]\subset\mathfrak{h}$, i.e, $G/H$ is symmetric.
\end{proof}

To prove Theorem \ref{main-thm-4} and Theorem \ref{main-thm-2}, we need the following lemma.

\begin{lemma}\label{lemma-1}
Let $(G,F)$ be a Finsler cyclic Lie group. Then we have the following:
\begin{enumerate}
\item If $y\in\g\backslash\{0\}$ satisfies $\langle[\mathfrak{g},\g],y\rangle_y=0$, then $\eta(y)=0$ and $N(y,\cdot)=-\mathrm{ad}(y)$. Moreover, in this situation, $R_y=-\mathrm{ad}(y)^2$ and $\mathrm{ad}(y)$ is self-adjoint with respect to $\langle\cdot,\cdot\rangle_y$.
\item Any $y\in\mathfrak{c}(\mathfrak{g})\backslash\{0\}$ satisfies $\langle[\mathfrak{g},\g],y\rangle_y=0$ and $R_y\equiv0$.
\end{enumerate}
\end{lemma}

\begin{proof}
(1) Since $\langle[\g,\g], y\rangle_y=0$, we have $\eta(y)=0$.
By the cyclic condition, we have
\begin{equation}\label{011}
\langle[u,y],v\rangle_y-\langle[v,y],u\rangle_y
=\langle[u,y],v\rangle_y+\langle[v,u],y\rangle_y+\langle[y,v],u\rangle_y=0,
\end{equation}
so (\ref{008}) provides
\begin{eqnarray*}
2\langle N(y,v),u\rangle_y&=&
\langle[u,v],y\rangle_y+\langle[u,y],v\rangle_y+\langle[v,y],u\rangle_y-2C_y(u,v,\eta(y))\\
&=&\langle[u,y],v\rangle_y+\langle[v,y],u\rangle_y
=2\langle[v,y],u\rangle_y.
\end{eqnarray*}
It implies $N(y,v)=[v,y]$, $\forall v\in\g$, i.e., $N(y,\cdot)=-\mathrm{ad}(y)$.
Input $\eta(y)=0$ and $N(y,\cdot)=-\mathrm{ad}(y)$ into (\ref{009}) in Theorem \ref{thm-homogeneous-Riemann-curvature}, we get $R_y=-\mathrm{ad}(y)^2$ immediately.
Finally, using (\ref{011}), we see that $N(y,\cdot)$ is self-adjoint with respect to $\langle\cdot,\cdot\rangle_y$.

(2) Assume $y\in\mathfrak{c}(\mathfrak{g})\backslash\{0\}$. By the cyclic condition, we have $\forall x,z\in\g$,
$$\langle[z,x],y\rangle_y=\langle[x,y],z\rangle_y+\langle[y,z],x\rangle_y+\langle[z,x],y\rangle_y=0,$$
i.e., $\langle[\g,\g],y\rangle_y=0$. By (1) of Lemma \ref{lemma-1}, $R_y=-\mathrm{ad}(y)^2=0$ follows immediately.
\end{proof}

%
%
\begin{proof}[Proof of Theorem \ref{main-thm-4}]
Assume conversely that $G$ admits a left invariant cyclic Finsler metric and its Lie algebra $\mathfrak{g}$ is non-Abelian.
By the proof of Theorem 5.1 in \cite{Hu2015-1}, there exists a nonzero $y\in\mathfrak{c}(\mathfrak{g})$ with $\mathrm{Ric}(y)=\mathrm{tr}(R_y)>0$.
This is a contradiction with (2) in Lemma \ref{lemma-1}.
\end{proof}

\begin{proof}[Proof of Theorem \ref{main-thm-2}]Assume $(G,F)$ is a Finsler cyclic Lie group, i.e., $F$ is a left invariant Finsler metric on $G$ which is cyclic with respect to the reductive decomposition $\g=\h+\m=0+\g$. We will prove $\g$ is Abelian by the following three claims.

{\bf Claim I}: $[\g,\g]$ is commutative.

The left invariance of $F$ implies that its Cartan tensor and Landsberg tensor are both bounded. So the argument proving Akbar-Zadeh's theorem \cite{Ak1988} can be applied here to prove that $F$ is locally Minkowskian. We average the fundamental tensor of $F$ on the indicatrix in each tangent space, with respect to the volume form induced by the Hessian metric, then we get a Riemannian metric $\overline{F}$ on $G$.
Because $F$ is locally Minkowskian, $\overline{F}$ is flat.
 This averaging process preserves Killing fields and isometries. So $\overline{F}$ is left $G$-invariant.
By Theorem 1.5 in \cite{Mi1976}, $\mathfrak{g}=\mathfrak{l}+\mathfrak{u}$, in which $\mathfrak{l}$ is a commutative subalgebra and $\mathfrak{u}$ is a commutative ideal. So we see
$[\g,\g]\subset\mathfrak{u}$ is commutative, which proves Claim I.

{\bf Claim II}: $\mathfrak{c}(\mathfrak{g})+[\g,\g]=\g$.

Assume conversely that $\mathfrak{c}(\mathfrak{g})+[\g,\g]\neq\g$. Then there exists a nonzero vector $y\in\g\backslash(\mathfrak{c}(\mathfrak{g})+[\g,\g])$, such that
$\langle\mathfrak{c}(\mathfrak{g})+[\g,\g],y\rangle_y=0$.
Because
$g_{y}([\mathfrak{g},\mathfrak{g}],y)=0$, we can apply the flatness condition for $F$ and (1) of Lemma \ref{lemma-1} to see that $\mathrm{ad}(y)$ is self adjoint with respect to $\langle\cdot,\cdot\rangle_{y}$ and $R_{y}(v)=-\mathrm{ad}(y)^2v=0$ for any $v\in\g$. Then we have
$$\langle \mathrm{ad}(y)v,\mathrm{ad}(y)v\rangle_{y}=
\langle\mathrm{ad}(y)^2v,v\rangle_{y}=
-\langle R_{y}(v),v\rangle_{y}=0,\quad\forall v\in\g,$$
i.e., $y\in\mathfrak{c}(\mathfrak{g})$. This is a contradiction, which proves Claim II.

{\bf Claim III}: $[\g,\g]\subset\mathfrak{c}(\g)$.

Choose any $u\in\g$ and $v\in[\g,\g]$.
By Claim II, we have the decomposition $u=u_1+u_2$ with $u_1\in\mathfrak{c}(\g)$ and $u_2\in[\g,\g]$. By Claim I, $[v,u_2]=0$, $[v,u]=[v,u_1]+[v,u_2]=0$. So we see $[\g,\g]\subset\mathfrak{c}(\g)$, which proves Claim III.

Summarizing Claim II and Claim III, we prove $\g=\mathfrak{c}(\g)$, which ends the proof.
\end{proof}

\section{Construction of cyclic $(\alpha,\beta)$ metrics}

To prove Theorem \ref{main-thm-3}, we need the following lemma.

\begin{lemma}\label{lemma-2}
Keeping all assumptions and notations in Theorem \ref{main-thm-3}, then
$F$ is cyclic if and only if $\Phi(r)\cdot \Psi(u,v,y)=0$ is satisfied for any $u,v\in\m$, $y\in\m\backslash\{0\}$, and $r=\tfrac{\alpha(X,y)}{\alpha(y,y)^{1/2}}$, in which
\begin{eqnarray*}\Phi(r)=\phi(r)\phi'(r)-r\phi'(r)^2-r\phi(r)\phi''(r)
\end{eqnarray*}
and
\begin{eqnarray*}
\Psi(u,v,y) &=&\alpha(y,[y,u]_\m)(\alpha(y,y)\alpha(X,v)-\alpha(y,v)\alpha(X,y))\\
                &   &+\alpha(y,[v,y]_\m)(\alpha(y,y)\alpha(X,u)-\alpha(y,u)\alpha(X,y)).
\end{eqnarray*}
\end{lemma}

\begin{proof}
Direct calculation for the fundamental tensor shows that, for any $y\in\m\backslash0$, $u,v\in\m$,
\begin{eqnarray}
& &\langle u,v\rangle_y =\tfrac{1}{2}\tfrac{\partial^2}{\partial s\partial t}|_{s=t=0}F^2(y+su+tv)
\nonumber\\
&=&\tfrac{1}{2}\tfrac{\partial^2}{\partial s\partial t}|_{s=t=0}\left(\alpha(y+su+tv,y+su+tv)
\phi(\tfrac{\alpha(X,y+su+tv)}{\sqrt{\alpha(y+su+tv,y+su+tv)}})^2\right)\nonumber\\
&=&\phi(r)^2\alpha(u,v)-\phi(r)\phi'(r)\cdot\tfrac{\alpha(y,u)\alpha(y,v)\alpha(X,y)}{\alpha(y,y)^{3/2}}\nonumber\\
& &+\phi(r)\phi'(r)\cdot\tfrac{\alpha(y,v)\alpha(X,u)+\alpha(y,u)\alpha(X,v)-\alpha(u,v)\alpha(X,y)
}{\alpha(y,y)^{1/2}}\nonumber\\
& &+(\phi'(r)^2+\phi(r)\phi''(r))(\alpha(X,u)-\tfrac{\alpha(y,u)\alpha(X,y)}{\alpha(y,y)})(\alpha(X,v)-\tfrac{\alpha(y,v)\alpha(X,y)}{\alpha(y,y)}).\label{013}
\end{eqnarray}
Setting
\begin{eqnarray*}
& &Q_1=\alpha(X,[y,u]_\m)\alpha(y,v), \quad Q_2=\alpha(X,[y,u]_\m), \quad
Q_3=\alpha(X,[u,v]_\m)\alpha(y,y),\\
& & Q_4=\alpha(X,[v,y]_\m)\alpha(y,u),\quad
Q_5=\alpha(X,[v,y]_\m),
\end{eqnarray*}
then we can get the following from (\ref{013}),
\begin{eqnarray}
\langle[y,u]_\m,v\rangle_y&=& \alpha([y,u]_\m,v)\phi(r)^2-
\phi(r)\phi'(r)\cdot\tfrac{\alpha(y,[y,u]_\m)\alpha(y,v)\alpha(X,y)}{\alpha(y,y)^{3/2}}\nonumber\\
& &+\phi(r)\phi'(r)\cdot\tfrac{Q_1+\alpha(y,[y,u]_{\m})\alpha(X,v)-
\alpha(v,[y,u]_\m)\alpha(X,y)}{\alpha(y,y)^{1/2}}\nonumber\\
& &+(\phi(r)\phi''(r)+\phi'(r)^2)\cdot(Q_2-\tfrac{\alpha(y,[y,u]_\m)\alpha(X,y)}{\alpha(y,y)})\nonumber\\
& &\cdot(\alpha(X,v)-\tfrac{\alpha(y,v)\alpha(X,y)}{\alpha(y,y)}),\label{014}\\
\langle[u,v]_\m,y\rangle_y
&=&\alpha([u,v]_\m,y)\phi(r)^2-\phi(r)\phi'(r)\cdot
\tfrac{\alpha(y,[u,v]_\m)\alpha(X,y)}{\alpha(y,y)^{1/2}}\nonumber\\
& &+\phi(r)\phi'(r)\cdot\tfrac{Q_3}{\alpha(y,y)^{1/2}},\label{015}\\
\langle[v,y]_\m,u\rangle_y&=&\alpha([v,y]_\m,u)\phi(r)^2-
\phi(r)\phi'(r)\cdot\tfrac{\alpha(y,[v,y]_\m)\alpha(y,u)\alpha(X,y)}{\alpha(y,y)^{3/2}}\nonumber\\
& &+\phi(r)\phi'(r)\cdot\tfrac{Q_4+\alpha(y,[v,y]_\m)\alpha(X,u)-
\alpha(u,[v,y]_\m)\alpha(X,y)}{\alpha(y,y)^{1/2}}\nonumber\\
& &+(\phi(r)\phi''(r)+\phi'(r)^2)\cdot(Q_5-\tfrac{\alpha(y,[v,y]_\m)\alpha(X,y)}{\alpha(y,y)})\nonumber\\
& &\cdot(\alpha(X,u)-\tfrac{\alpha(y,u)\alpha(X,y)}{\alpha(y,y)})\label{016}.
\end{eqnarray}
The cyclic condition for $\alpha$ implies
$$\alpha([y,u]_\m,v)+\alpha([u,v]_\m,y)+\alpha([v,y]_\m,u)=0.$$
Since $F$ is a non-Riemannian Douglas metric, we have $\alpha([\m, \m]_{\m}, X)=0$ (see \cite{AD2008} or \cite{LD2015}), so $Q_1=Q_2=Q_3=Q_4=Q_5=0$. Using these observations, the sum of (\ref{014})-(\ref{016}) provides
\begin{eqnarray*}
& &\langle[y,u]_\m,v\rangle_y+\langle[u,v]_\m,y\rangle_y+\langle[v,y]_\m,u\rangle_y\\
&=&
\phi(r)\phi'(r)\cdot\tfrac{\alpha(y,[y,u]_\m)(\alpha(y,y)\alpha(X,v)-\alpha(y,v)\alpha(X,y))
+\alpha(y,[v,y]_\m)(\alpha(y,y)\alpha(X,u)-\alpha(y,u)\alpha(X,y))}{\alpha(y,y)^{3/2}}\\
& &-(\phi(r)\phi''(r)+\phi'(r)^2)\tfrac{\alpha(X,y)}{\alpha(y,y)^{1/2}}\cdot\\
& &\tfrac{\alpha(y,[y,u]_\m)(\alpha(y,y)\alpha(X,v)-\alpha(y,v)\alpha(X,y))
+\alpha(y,[v,y]_\m)(\alpha(y,y)\alpha(X,u)-\alpha(y,u)\alpha(X,y))}{\alpha(y,y)^{3/2}}\\
&=&\tfrac{1}{\alpha(y,y)^{3/2}}\cdot\Phi(r)\cdot\Psi(u,v,y).
\end{eqnarray*}
So $F$ is cyclic, i.e., $\langle[y,u]_\m,v\rangle_y+\langle[u,v]_\m,y\rangle_y+\langle[v,y]_\m,u\rangle_y\equiv0$,
if and only if $\Phi(r)\cdot\Psi(u,v,y)\equiv0$, which proves the lemma.
\end{proof}

Now we prove Theorem \ref{main-thm-3}, i.e., $F$ is cyclic if and only if $\Psi(u,v,y)\equiv0$.

\begin{proof}[Proof of Theorem \ref{main-thm-3}]Assume $\Psi(u,v,y)\equiv0$, then
Lemma \ref{lemma-2} indicates $F$ is cyclic. This proves one side of the theorem.

Assume $F$ is cyclic. By Lemma \ref{lemma-2}, we have
$\Phi(r)\cdot \Psi(u,v,y)\equiv0$. Notice that $\Phi(r)$ can not be constantly zero when $|r|\leq\alpha(X,X)^{1/2}$, otherwise
$\Phi(r)=(\phi(r)(\phi(r)-r\phi'(r)))'\equiv0$, i.e., $\phi(r)^2-r\phi\phi'(r)\equiv c$ for some constant $c$,
which can be easily solved and provides the solutions
$\phi(r)\equiv\sqrt{c_1r^2+c_2}$ for some constants $c_1$ and $c_2$. Then the metric $F=\alpha\phi(\tfrac{\beta}{\alpha})$ is Riemannian, which contradicts our assumption.
To summarize, there exists a nonempty open subset $\mathcal{U}\subset\m\backslash\{0\}$, such that
$\Psi(u,v,y)=0$ when $y\in\mathcal{U}$. Since $\Psi(u,v,y)$ is a polynomial, it must vanish identically.
This ends the proof for the other side of Theorem \ref{main-thm-3}.
\end{proof}

Now we use Theorem \ref{main-thm-3} to construct a left invariant cyclic Randers metric.

Let $G$ be a solvable Lie group. Its Lie algebra $\g$ has a chain of ideals,
$$\g=\g_0\supset \g_1\supset\cdots\supset \g_n=0,$$
in which $\dim\g_i=n-i$, $\forall 0\leq i\leq n$. We can find a basis $\{e_1,\cdots,e_n\}$ of $\g$,
such that $\g_{n-i}=\mathrm{span}\{e_1,\cdots,e_i\}$ for each $i$, and then an inner product $\alpha(\cdot,\cdot)$ on $\g$, such that $\langle e_i,e_j\rangle=\delta_{ij}$.
The following proposition in \cite{GGO2015} answers exactly when this $\alpha$ induces a left invariant cyclic metric on $G$.
%
%
\begin{Prop}\label{prop1.8}
$(G,\alpha)$ is cyclic if and only if $\mathrm{ad}(e_i)$ is self adjoint on $\g_{n-i}$ for each $i$, or equivalently, the coefficients $c_{ij}^k$ in $[e_i,e_j]=c_{ij}^ke_k$ satisfy $c_{ik}^j=c_{jk}^i$, $\forall 1\leq i<j<k\leq n$.
\end{Prop}

According to Proposition \ref{prop1.8}, $(G,\alpha)$ is cyclic when
 \begin{eqnarray*}\label{equ17}
 [e_i, e_j]=0,  1\leq i<j\leq n-1,\quad   [e_n, e_i]=a e_i,  1\leq i \leq n-1,
 \end{eqnarray*}
in which $a$ is a positive number.
We choose $X=ce_n$ for some $c\in(-1,1)\backslash\{0\}$ and $\beta=\langle X,\cdot\rangle$. Then $F=\alpha+\beta$ induces a left invariant Randers metric on $G$. It is a non-Riemannian Douglas metric because $\alpha(X,[\g,\g])=0$ \cite{AD2008,De2012}.
Calculation shows that for any generic $u,v,y$,
$$\tfrac{\alpha(y,y)\alpha(X,u)-\alpha(X,y)\alpha(u,y)}{\alpha([y,u],y)}
=-\tfrac{c}{a}=\tfrac{\alpha(y,y)\alpha(X,v)-\alpha(X,y)\alpha(v,y)}{\alpha([y,v],y)},$$
i.e., $\alpha$ satisfies the equation in Theorem \ref{main-thm-3}, so $F$ is cyclic. Using the same $\alpha$ and $\beta$, left invariant cyclic $(\alpha,\beta)$ metrics can be similarly constructed.

{\bf Acknowledgement}.
This paper is supported by National Natural Science Foundation of China (No. 12001007, No. 12131012, No. 11821101), Beijing Natural Science Foundation (No. 1222003), Natural Science Foundation of Anhui province (No. 1908085QA03).

\end{document}